\documentclass{amsart}
\usepackage{graphicx}
\usepackage{amsmath}
\usepackage{amssymb}
\usepackage{amsthm}
\usepackage{esint}

\newtheorem{theorem}{Theorem}[section]
\newtheorem{lem}[theorem]{Lemma}

\theoremstyle{Corollary}

\newtheorem{prop}[theorem]{Proposition}

\theoremstyle{definition}
\newtheorem{definition}[theorem]{Definition}
\newtheorem{exmp}{Example}[section]

\numberwithin{equation}{section}

\makeatletter
\newcommand*\owedge{\mathpalette\@owedge\relax}
\newcommand*\@owedge[1]{%
  \mathbin{%
    \ooalign{%
      $#1\m@th\bigcirc$\cr
      \hidewidth$#1\m@th\wedge$\hidewidth\cr
    }%
  }%
}
\makeatother

\newcommand{\gone}{g^{(1)}}
\newcommand{\gtwo}{g^{(2)}}

\begin{document}

\title{Convergence rate of the $Q$-curvature flow}

\author{Pak Tung Ho}
\address{Department of Mathematics, Tamkang University Tamsui, New Taipei City 251301, Taiwan}

\email{paktungho@yahoo.com.hk}

\author{Sanghoon Lee}
\address{Korea Institute for Advanced Study, Hoegiro 85, Seoul 02455, Korea}
\email{sl29@kias.re.kr}


\date{5th April, 2024.}


\begin{abstract}
Carlotto, Chodosh and Rubinstein have studied the convergence rate of the Yamabe flow.
Inspired by their result, we study the convergence rate of the $Q$-curvature flow in this paper. In particular, we provide an example of a  slowly converging $Q_6$-curvature flow in dimension 6, in constrast to the dimension 2 case, where the $Q$-curvature flow always converges exponentially.

\end{abstract}

\maketitle

\section{Introduction}

Let $(M,g_0)$ be a closed (i.e. compact without boundary) Riemannian manifold
of dimension $n\geq 3$. As a generalization of the Uniformization Theorem,
the \textit{Yamabe problem} is to find a metric conformal to $g_0$ such that
its scalar curvature $R_g$ is constant.
This was solved by Aubin \cite{Aubin0}, Trudinger \cite{Trudinger} and Schoen \cite{Schoen}.

The \textit{Yamabe flow} is a geometric flow introduced to
study the Yamabe problem. It is defined as
\begin{equation}\label{1.4}
\frac{\partial}{\partial t}g(t)=-(R_{g(t)}-\overline{R}_{g(t)})g(t),
\end{equation}
where $\overline{R}_{g(t)}$ is the average of the scalar curvature of $g(t)$:
\begin{equation*}
\overline{R}_{g(t)}=\frac{\int_MR_{g(t)}dV_{g(t)}}{\int_MdV_{g(t)}}.
\end{equation*}
The existence and convergence of the Yamabe flow
has been studied in \cite{Brendle4,Brendle5,Chow,Schwetlick&Struwe,Ye}.
See also  \cite{Azami&Razavi,Cheng&Zhu,Daneshvar&Razavi,Ho1,Ho2,Ho3,Ma&Cheng,Schulz} and references therein
for results related to the Yamabe flow.

In \cite{Chodosh}, Carlotto, Chodosh and Rubinstein
studied the rate of convergence of the Yamabe flow (\ref{1.4}).
They proved the following: (see Theorem 1 in \cite{Chodosh})

\begin{theorem}\label{thm1}
Assume $g(t)$ is a solution of the Yamabe flow \eqref{1.4}
that converges in $C^{2,\alpha}(M,g_\infty)$ to $g_\infty$ as $t\to\infty$ for some $\alpha\in (0,1)$.
Then there is a $\delta>0$ depending only on $g_\infty$ such that:\\
(i) If $g_\infty$ is an integrable critical point, then the convergence occurs at an exponential rate, that is
$$\|g(t)-g_\infty\|_{C^{2,\alpha}(M,g_\infty)}\leq Ce^{-\delta t}$$
for some constant $C>0$ depending on $g(0)$. \\
(ii) In general, the rate of convergence cannot be worse than polynomial, that is
$$\|g(t)-g_\infty\|_{C^{2,\alpha}(M,g_\infty)}\leq C(1+t)^{-\delta}$$
for some constant $C>0$ depending on $g(0)$.
\end{theorem}

\begin{theorem}\label{thm2}
Assume that $g_\infty$ is a nonintegrable critical point of the Yamabe energy with
order of integrability $p\geq 3$. If $g_\infty$ satisfies the Adams-Simon positive condition
$AS_p$, then there exists metric $g(0)$ conformal to $g_\infty$ such that
the solution $g(t)$ of the Yamabe flow \eqref{1.4}
starting from $g(0)$ exists for all time and converges
in $C^\infty(M,g_\infty)$ to $g_\infty$
as $t\to\infty$. The convergence occurs ``slowly" in the sense that
$$C^{-1}(1+t)^{-\frac{1}{p-1}}\leq \|g(t)-g_\infty\|_{C^{2,\alpha}(M,g_\infty)}\leq C(1+t)^{-\frac{1}{p-2}}$$
for some constant $C>0$.
\end{theorem}

We refer the readers to \cite[Definition 8]{Chodosh}
and \cite[Definition 10]{Chodosh}
respectively for the precise definitions of integrable critical point
and Adams-Simon positive condition $AS_p$.
In \cite{Chodosh}, examples of $g_\infty$ satisfying the $AS_p$,
for some $p\geq 3$, have been found. Hence,
there exist examples of the Yamabe flow \eqref{1.4}
which does not converge exponentially.
To the best of the authors' knowledge,
this is the first geometric flow
found not to have exponential convergence.
Inspired by   \cite{Chodosh},
the convergence rate of other geometric flows have been studied.
In \cite{Ho&Shin}, the convergence rates of the Yamabe flow with boundary
and the conformal mean curvature flow on manifolds with boundary have been studied.
In \cite{Ho&Shin&Yan}, the convergence rate of the weighted
Yamabe flow on the smooth metric measure spaces has been studied.

Suppose that $(M,g)$ is a closed $4$-dimensional Riemannian manifold.
The \textit{$Q$-curvature} is defined by
\begin{equation}\label{q}
Q_g=-\frac{1}{6}(\Delta_g R_g-R_g^2+3|Ric_g|^2),
\end{equation}
where $R_g$ is the scalar curvature
and $Ric_g$ is the Ricci curvature tensor.
The \textit{Paneitz operator} is defined by
\begin{equation}\label{p}
P_gf=\Delta_g^2f+\delta\left[\left(\frac{2}{3}R_gg-2Ric_g\right)df\right]
\end{equation}
for $f\in C^\infty(M)$.
Here $\delta$ denotes the divergence and $d$ the differential.
Under a conformal change of the metric $\tilde{g}=e^{2w}g$,
the $Q$-curvature transforms according to
\begin{equation}\label{1.6}
P_gw+Q_g=Q_{\tilde{g}}e^{4w}.
\end{equation}
The Chern-Gauss-Bonnet theorem asserts that
$$\int_MQ_gdV_g+\int_M\frac{1}{4}|W_g|^2dV_g=8\pi^2\chi(M)$$
where $\chi(M)$ is the Euler-characteristics of $M$,
and $W_g$ is the Weyl tensor.
Note that $W_g$ is pointwise conformally invariant in the sense that
$|W_{\tilde{g}}|^2dV_{\tilde{g}}=|W_g|^2dV_g$
whenever $\tilde{g}=e^{2w}g$.
Hence, the total $Q$-curvature
$$\int_MQ_gdV_g$$
is conformally invariant.

Now suppose that $(M,g)$ is a closed Riemannian manifold of dimension $n$, where $n$ is even.
In \cite{Fefferman&Graham1,Fefferman&Graham2},
Fefferman and Graham constructed a conformally invariant formally invariant self-adjoint operator
$P_g$ with the leading term $(-\Delta_g)^{\frac{n}{2}}$.
Moreover, there is a curvature quantity $Q_g$ which transforms according to
\begin{equation}\label{1.12}
P_gw+Q_g=Q_{\tilde{g}}e^{nw}
\end{equation}
whenever $\tilde{g}=e^{2w}g$.
This implies that the   total $Q$-curvature
$$\int_MQ_gdV_g$$
is conformally invariant.
Indeed, we have
$$P_g=-\Delta_g~~\mbox{ and }~~Q_g=K_g$$
when $n=2$, where $K_g$ is the Gauss curvature.
In particular, it follows from the Gauss-Bonnet theorem that
$$\int_MK_gdV_g=2\pi\chi(M)$$
is a conformally invariant.
Moreover, when $n=4$, the operator $P_g$ and
the curvature quantity $Q_g$ reduce to
the Paneitz operator and the $Q$-curvature defined
in (\ref{p}) and (\ref{q}) respectively.

Uniformization theorem for $Q$-curvature
and the problem of prescribing $Q$-curvature have been studied.
See \cite{Brendle1,Chang&Yang,Chen&Xu,Djadli&Malchiodi,Ho5,Ho7,Ho6,Malchiodi&Struwe,Ndiaye,Wei&Xu}
and the references therein
for the related results.
In particular, the \textit{$Q$-curvature flow} has been defined to study
the Uniformization theorem for $Q$-curvature:
\begin{equation}\label{1.1}
\frac{\partial}{\partial t}g(t)=-(Q_{g(t)}-\overline{Q}_{g(t)})g(t),
\end{equation}
where $Q_{g(t)}$ is the $Q$-curvature of $g(t)$, and $\overline{Q}_{g(t)}$ is its average:
\begin{equation}\label{1.2}
\overline{Q}_{g(t)}=\frac{\int_M Q_{g(t)}dV_{g(t)}}
{\int_MdV_{g(t)}}.
\end{equation}
The $Q$-curvature (\ref{1.1}) on the sphere $S^n$ has been studied in \cite{Brendle}
when $n=4$
and in  \cite{Ho4} for general $n$.
In particular,
Brendle proved in \cite{Brendle1} the following:

\begin{theorem}\label{Brendle_thm}
Assume that the Paneitz operator of $(M,g)$ is weakly positive with kernel
consisting of the constant functions. Moreover, assume that
$$\int_MQ_gdV_g<(n-1)!\omega_n$$
where $\omega_n$ is the volume of standard sphere $S^n$.
Then the $Q$-curvature flow \eqref{1.1} exists for all times
and converges to a metric of constant $Q$-curvature.
\end{theorem}

We remark that \cite[Theorem 1.1]{Brendle1} is more general than
Theorem \ref{Brendle_thm}. Indeed, Brendle in \cite{Brendle1} proved the
existence and convergence of the prescribed $Q$-curvature flow
under the same assumptions in Theorem \ref{Brendle_thm}.

Inspired
by the results of Carlotto, Chodosh and Rubinstein about the convergence rate of Yamabe flow,
i.e. Theorems \ref{thm1} and \ref{thm2} mentioned above,
 we study in this paper the convergence rate of the $Q$-curvature flow (\ref{1.1}).

The following theorems are the main results in this paper:

\begin{theorem}\label{main1}
Assume that $g(t)$ is a solution to the $Q$-curvature flow
that is converging in $C^{n,\alpha}(M,g_\infty)$ to $g_\infty$
as $t\to\infty$ for some $\alpha\in (0,1)$. Then, there is $\delta>0$ depending only on $f$
and $g_\infty$ such that\\
(i) If $g_\infty$ is an integrable critical point, then the convergence
occurs at an exponential rate
$$\|g(t)-g_\infty\|_{C^{n,\alpha}(M,g_\infty)}\leq Ce^{-\delta t},$$
for some constant $C>0$ depending on $g(0)$.\\
(ii) In general, the convergence cannot be worse than a polynomial rate
 $$\|g(t)-g_\infty\|_{C^{n,\alpha}(M,g_\infty)}\leq C(1+t)^{-\delta},$$
for some constant $C>0$ depending on $g(0)$.
\end{theorem}

\begin{theorem}\label{main2}
Assume that $g_\infty$ is a non-integrable critical point of
the functional $E$ with order of integrability $p\geq 3$. If $g_\infty$
satisfies the Adam-Simon positivity condition $AS_p$, then there exists
a metric $g(0)$ conformal to $g_\infty$
such that the $Q$-curvature flow
$g(t)$ starting from $g(0)$ exists for all time and converges in $C^\infty(M,g_\infty)$ to $g_\infty$
as $t\to\infty$. The convergence occurs ``slowly" in the sense that
$$C(1+t)^{-\frac{1}{p-2}}\leq \|g(t)-g_\infty\|_{C^{n}(M,g_\infty)}\leq C(1+t)^{-\frac{1}{p-2}}$$
for some constant $C>0$.
\end{theorem}

The precise definitions of integrable critical point and
the Adam-Simon positivity condition $AS_p$ can be found in Section \ref{section2}.
In Section \ref{section3}, we prove Theorem \ref{main1}, 
In Section \ref{section4}, we prove Theorem \ref{main2}.
In Section \ref{section5}, we construct an example of Riemannian manifold
that satisfies the condition $AS_3$.
This allows us to conclude that
there exists a $Q$-curvature flow
converges  exactly at a polynomial rate described in Theorem \ref{main2}. 
In Section \ref{section6}, we compare low dimensional cases when $n= 2, 4, 6$.

\section{The $Q$-curvature flow}\label{section2}

Suppose that $n$ is even.
Let $M$ be a compact $n$-dimensional manifold equipped with the Riemannian metric $g_\infty$
such that its volume satisfies
$\displaystyle\int_MdV_{g_\infty}=1.$
Along the $Q$-curvature flow (\ref{1.1}), the volume
$\displaystyle\int_M dV_{g(t)}$ is preserved.
Indeed,
it follows from (\ref{1.1}) and (\ref{1.2}) that
\begin{equation}\label{1.3}
\frac{d}{dt}\left(\int_M dV_{g(t)}\right)=
\int_M \frac{\partial}{\partial t}(dV_{g(t)})
=-\frac{n}{2}\int_M (Q_{g(t)}-\overline{Q}_{g(t)})dV_{g(t)}=0.
\end{equation}
Hence, if we assume $\displaystyle\int_M dV_{g(0)}=1$, then
$$\int_M dV_{g(t)}=1~~\mbox{ for all }t\geq 0.$$
Since the flow (\ref{1.1}) preserves the conformal structure,
we can write $g(t)=e^{2u(t)}g_\infty$.
Therefore, (\ref{1.1}) can be written
as
\begin{equation}\label{1.8}
\frac{\partial}{\partial t}u(t)=-\frac{1}{2}(Q_{g(t)}-\overline{Q}_{g(t)}).
\end{equation}
Since $g(t)=e^{2u(t)}g_\infty$, it follows from (\ref{1.12}) that
\begin{equation}\label{1.11}
P_{g_\infty}u(t)+Q_{g_\infty}=Q_{g(t)}e^{nu(t)}.
\end{equation}
We define the functional
$$E(w)=\int_M\left(\frac{n}{2}wP_{g_\infty}w+n Q_{g_\infty} w\right)dV_{g_\infty}
-\left(\int_MQ_{g_\infty} dV_{g_\infty}\right)\log\left(\int_Me^{nw}dV_{g_\infty}\right).$$
Using these, we find
\begin{equation}\label{1.5}
\frac{d}{dt}E(u(t))
=-\frac{n}{2} \int_M(Q_{g(t)}-\overline{Q}_{g(t)})^2dV_{g(t)}
\end{equation}
along the $Q$-curvature flow.

Consider the  \textit{unit volume conformal class} associated to $g_\infty$:
$$[g_\infty]_1=\Big\{e^{2w}g_\infty: w\in C^{n,\alpha}(M), \int_M e^{nw}dV_{g_\infty}=1\Big\}.$$
For $k\in\mathbb{N}$, we denote the $k$-th differential of the functional $E$ on $[g_\infty]_1$
at the point $w$ in the directions $v_1,..., v_k$ by
$$D^kE(w)[v_1,..., v_k].$$
As we will see below, the functional $v\mapsto D^kE(w)[v_1,..., v_{k-1},v]$
is in the image of $L^2(M,g_\infty)$ under the natural embedding onto $C^{n,\alpha}(M,g_\infty)'$.
Therefore, we will also write
$$D^kE(w)[v_1,..., v_{k-1}]$$
for this element of $L^2(M,g_\infty)$.
When $k=1$, we will drop the (second) brackets, and thus consider
$Df(w)\in L^2(M,g_\infty)$.

We may write the differential of $E$ restricted to $[g_\infty]_1$ as
\begin{equation*}
\begin{split}
\frac{1}{2}DE(w)[v]
&=n\int_M\big(P_{g_\infty}w+Q_{g_\infty}-\overline{Q}_{g_\infty}e^{nw}\big)vdV_{g_\infty}\\
&=n\int_M\big(Q_{g_w}-\overline{Q}_{g_w}\big)e^{nw}vdV_{g_\infty}
\end{split}
\end{equation*}
where $g_w=e^{2w}g_\infty$, where we have used the facts that
$\displaystyle\int_M dV_{g_\infty}=\int_M dV_{g_w}=1$
and
the total $Q$-curvature is conformally invariant:
$$\int_MQ_{g_\infty}dV_{g_\infty}=\int_MQ_{g_w}dV_{g_w}.$$
Thus, a metric $g_\infty$ is a critical point for the energy
$E$ restricted to $[g_\infty]_1$ exactly when
$g_\infty$ has constant $Q$-curvature.
Regarded as an element of $L^2(M,g_\infty)$, we have that
\begin{equation}\label{eq2}
  \frac{1}{2}DE(w)=P_{g_\infty}w+Q_{g_\infty}-\overline{Q}_{g_\infty}e^{nw}.
\end{equation}

From now on, we assume that
$g_\infty$ has constant $Q$-curvature.
We denote by $\mathcal{CQC}_1$
the set of unit volume, constant $Q$-curvature metrics in $[g_\infty]_1$.
We find
\begin{equation}\label{2.3}
\frac{1}{2}D^2E(g_\infty)[v,w]
=\frac{1}{2}\left.\frac{d}{dt}\left(DE(tv)[w]\right)\right|_{t=0}=n\int_M\big(P_{g_\infty}v-nQ_{g_\infty}v\big)wdV_{g_\infty}.
\end{equation}
Therefore, if we define $\mathcal{L}_\infty$ by means of the formula
\begin{equation*}
\frac{1}{n}\int_Mw\mathcal{L}_\infty v dV_{g_\infty}
:=\frac{1}{2}D^2E(g_\infty)[v,w],
\end{equation*}
then, by  (\ref{2.3}), we can obtain
$$\mathcal{L}_\infty v=P_{g_\infty}v-nQ_{g_\infty}v.$$
We define $\Lambda_0:=\ker\mathcal{L}_\infty\subset L^2(M,g_\infty)$.

Note that $P_{g_\infty}-nQ_{g_\infty}$
is a self-adjoint operator with leading term $(-\Delta_{g_\infty})^{\frac{n}{2}}$.
It follows from a classical theorem of spectral theory (cf. \cite{Lablee}) that $\Lambda_0$ is finite dimensional,
since it is the  eigenspace of the operator $P_{g_\infty}-nQ_{g_\infty}$
for the zero eigenvalue. We will write $\Lambda_0^{\perp}$
for the $L^2(M,g_\infty)$-orthogonal complement.

\medskip

It is crucial throughout this work that the functional $E$ is an analytic map in the sense of  \cite[Definition 8.8]{Zeidler}. More precisely,
 fix a metric $g_\infty$ then  the functional $E$ is an analytic functional on $C^{n,\alpha}(M,g_\infty)$ in the sense that for each $w_0\in C^{n,\alpha}(M,g_\infty)$, there is an $\epsilon>0$ and bounded multilinear operators
  \begin{equation*}
    E^{(k)}:C^{n,\alpha}(M,g_\infty)^{\times k}\rightarrow \mathbb{R}\textrm{ for each }k\geq 0
  \end{equation*}
  such that if $\|w-w_0\|_{C^{n,\alpha}}<\epsilon$, then $\sum_{k=0}^\infty\|E^{(k)}\|\cdot\|w-w_0\|^k_{C^{2n,\alpha}}<\infty$ and
  \begin{equation*}
    E(w)=\sum_{k=0}^\infty E^{(k)}(\underbrace{w-w_0,\cdots,w-w_0}_{k\textrm{-times}})\textrm{ in }C^{n,\alpha}(M,g_\infty).
  \end{equation*}
It is not hard to verify this, by simply expanding the denominator of $E$ in a power
series around $\displaystyle\log\left(\int_Me^{nw_0}dV_{g_\infty}\right)$
 and noting that the other terms in $E$ are already linear or bilinear
functions in $w$.

\medskip
We need the following Proposition from \cite[Section 3]{Simon}; which can be established with the help of the implicit function theorem:

\begin{prop}\label{prop7}
There is $\epsilon>0$ and an analytic map $\Phi:\Lambda_0\cap \{v: \|v\|_{L^2}<\epsilon\}
\to C^{n,\alpha}(M,g_\infty)\cap \Lambda_0^\perp$
such that $\Phi(0)=0$, $D\Phi(0)=0$,
\begin{equation}\label{eq2.7}
\sup_{\substack{
\|v\|_{L^2}<\epsilon,\\
\|w\|_{L^2}\leq 1}}\|D\Phi(v)[w]\|_{L^2}<1,
\end{equation}
and so that defining $\Psi(v)=v+\Phi(v)$,
we have that  $\displaystyle\int_M e^{2\Psi(v)}dV_{g_\infty}=1$ and
$$\mbox{\emph{proj}}_{\Lambda_0^\perp}[DE(\Psi(v))]=
\mbox{\emph{proj}}_{\Lambda_0^\perp}
\left[\Big(Q_{e^{2\Psi(v)}g_\infty}-\overline{Q}_{e^{2\Psi(v)}g_\infty}\Big)e^{n\Psi(v)}\right]=0.$$
Furthermore
$$\mbox{\emph{proj}}_{\Lambda_0}[DE(\Psi(v))]=
\mbox{\emph{proj}}_{\Lambda_0}
\left[\Big(Q_{e^{2\Psi(v)}g_\infty}-\overline{Q}_{e^{2\Psi(v)}g_\infty}\Big)e^{n\Psi(v)}\right]=DF,$$
where $F:\Lambda_0\cap\{v: \|v\|_{L^2}\leq \epsilon\}\to\mathbb{R}$ is defined by
$F(v)=E(\Psi(v))$. Finally, the intersection of $\mathcal{CQC}_1$ with a small $C^{n,\alpha}(M,g_\infty)$-neighborhood
of $0$ coincides with
$$\mathcal{S}_0:=\{\Psi(v): v\in\Lambda_0, \|v\|_{L^2}<\epsilon, DF(v)=0\},$$
which is a real analytic subvariety (possible singular) of the following
$(\dim\Lambda_0)$-dimensional real analytic submanifold of $C^{n,\alpha}(M,g_\infty)$:
$$\mathcal{S}:\{\Psi(v): v\in\Lambda_0, \|v\|_{L^2}<\epsilon\}.$$
\end{prop}

We will refer to $\mathcal{S}$ as the \textit{natural constraint} for the problem.

\begin{definition}\label{def8}
  For $g_\infty\in \mathcal{CQC}_1$, we say that $g_\infty$ is \textit{integrable} if for all $v\in\Lambda_0$, there is a path $w(t)\in C^n((-\epsilon,\epsilon)\times M,g_\infty)$ such that $e^{2w(t)}g_\infty\in \mathcal{CQC}_1$ and $w(0)=0$, $w'(0)=v$. Equivalently, $g_\infty$ is integrable if and only if $\mathcal{CQC}_1$ agrees with $\mathcal{S}$ in a small neighborhood of $0$ in $C^{n,\alpha}(M,g_\infty)$.
\end{definition}

We remark that the integrability defined in Definition \ref{def8} is equivalent to the functional $F$ (as defined in Proposition \ref{prop7}) being constant in a neighborhood of $0$ inside $\Lambda_0$ \cite[Lemma 1]{Adams&Simon}.

\begin{definition}
  If $\Lambda_0=0$, i.e. if $\mathcal{L}_\infty$ is injective, then we call $g_\infty$ a \textit{nondegenerate critical point}. On the other hand, if $\Lambda_0$ is nonempty, ewe call $g_\infty$ \textit{degenerate}.
\end{definition}

Note that if $g_\infty$ is a nondegenerate critical point, then $g_\infty$ is automatically integrable in the above sense.

\medskip

Now suppose that $g_\infty$ is a nonintegrable critical point. Because $F(v)=E(\Psi(v))$, defined in Proposition \ref{prop7},  is analytic, we may expand it in a power series
\begin{equation*}
  F(v)=F(0)+\sum_{j\geq p}F_j(v)
\end{equation*}
where $F_j$ is a degree-$j$ homogeneous polynomial on $\Lambda_0$ and $p$ is chosen so that $F_p$ is nonzero. We will call $p$ the \textit{order of integrability} of $g_\infty$. We will also need a further hypothesis for nonintegrable critical points introduced in \cite{Adams&Simon}.

\begin{definition}\label{def10}
  We say that $g_\infty$ satisfies the \textit{Adams-Simon positivity condition}, $AS_p$ for short (here $p$ is the order of integrability of $g_\infty$), if it is nonintegrable and $F_p|_{\mathbb{S}^k}$ attains a positive maximum for some $\hat{v}\in \mathbb{S}^k\subset \Lambda_0$. Recall that $F_p$ is the lowest-degree nonconstant term in the power series expansion of $F(v)$ around $0$ and $\mathbb{S}^k$ is the unit sphere in $\Lambda_0$.
\end{definition}

An important observation is that when the order of integrability $p$ is odd, the Adams-Simon positivity condition  is always satisfied. Moreover the order of integrability (at a critical point of $E$) always satisfies $p\geq 3$. In the next Proposition, we compute $F_3$ term.

\begin{prop}
We have the following identity:
\begin{equation}\label{1.10}
F_3(v)=-2n^2
\int_MQ_{g_\infty} v^3 dV_{g_\infty}.
\end{equation}
\end{prop}

\begin{proof}
We follow the computation in \cite[Appendix A]{Chodosh}.
It is easy to see that $F_1(v)=F_2(v)=0$.
To compute $D^3F(0)$, we may in fact compute $D^3\tilde{F}(0)$,
where $\tilde{F}:\Lambda_0\to\mathbb{R}$ is defined by
$\tilde{F}(v)=E(v)$. We first compute $D^3F$:
\begin{equation}\label{4.2}
\begin{split}
D^3F(w)[v,u,z]&=D^3E(\Psi(w))[D\Psi(w)[v],D\Psi(w)[u],D\Psi(w)[z]]\\
&\quad+D^2E(\Psi(w))[D^2\Psi(w)[u,z],D\Psi(w)[v]]\\
&\quad+D^2E(\Psi(w))[D\Psi(w)[u],D^2\Psi(w)[v,z]]\\
&\quad+D^2E(\Psi(w))[D\Psi(w)[z],D^2\Psi(w)[v,u]]\\
&\quad+\langle DE(\Psi(w)), D^3\Psi(w)[v,u,z]\rangle.
\end{split}
\end{equation}
Setting $w=0$, and using similar considerations as before (in particular noting that
$D^2E(0)[\cdot]$ is self-adjoint), we obtain $D^3F(0)[v,u,z]=D^3E(0)[v,u,z]$.
Performing the same computation for $D^3\tilde{F}(0)$ yields the same result.
Next, we compute $D^3\tilde{F}(0)$.
Recall from (\ref{eq2}) that
$$\frac{1}{2}DE(w)[v]=
\int_M(P_{g_\infty}w+Q_{g_\infty}-\overline{Q}_{g_\infty}e^{nw})vdV_{g_\infty}.$$
When computing the third derivative of $E$ at $0$, the first two terms will
vanish. Hence, we have
\begin{equation}\label{A.1}
\begin{split}
D^3E(0)[v,u,z]
&
=-2n^2
\int_MQ_{g_\infty} vuz dV_{g_\infty}.
\end{split}
\end{equation}
This together with (\ref{4.2}) implies that
$$D^3F(0)[v,u,z]=-2n^2
\int_MQ_{g_\infty} vuz dV_{g_\infty},$$
which proves (\ref{1.10}).
\end{proof}

\section{Proof of Theorem \ref{main1}} \label{section3}

One of the tools for controlling the rate of convergence of the $Q$-curvature flow will be the {\L}ojasiewicz-Simon inequality.

\begin{prop}\label{prop13}
Suppose that $g_\infty\in CQC_1$. There is $\theta\in(0,\frac{1}{2}]$,
$\epsilon>0$ and $C>0$ (both depending only on $n$ and $g_\infty$)
such that for $u\in C^{n,\alpha}(M,g_\infty)$ with
$\|u\|_{C^{n,\alpha}(M,g_\infty)}<\epsilon$, then
$$|E(u)-E(0)|^{1-\theta}\leq C\|DE(u)\|_{L^2(M,g_\infty)}.$$
If $g_\infty$ is an integrable critical point, then $\theta=\frac{1}{2}$.
If $g_\infty$ is non-integrable,
then this holds for some $\theta\in(0,\frac{1}{p}]$, where $p$ is the order of integrability of
$g_\infty$.
\end{prop}
\begin{proof}
  One may follow the argument in \cite{Chodosh}.
\end{proof}

\begin{proof}[Proof of Theorem \ref{main1}]
We consider the $Q$-curvature flow $g(t)=e^{2u(t)}g_\infty$
which converges to $g_\infty$ in $C^n(M,g_\infty)$ as $t\to\infty$.
In Proposition \ref{prop13}, we have shown that there is a {\L}ojasiewicz-Simon
inequality near $g_\infty$ for some $\theta\in (0,\frac{1}{2}]$.
We emphasize that if we regard $DE(g(t))$ as an element of $L^2(M,g_\infty)$, then
\begin{equation}\label{1.9}
DE(g(t))=2\big(Q_{g(t)}-\overline{Q}_{g(t)}\big)e^{nu(t)}.
\end{equation}
Since $g(t)\to g_\infty$ as $t\to\infty$,
we can choose $t_0$ so that for $t\geq t_0$, $\|u(t)\|_{C^0(M,g_\infty)}\leq \frac{1}{2}$.
This together with (\ref{1.5}) and Proposition \ref{prop13} implies that
\begin{equation}\label{1.7}
\begin{split}
\frac{d}{dt}\big(E(u(t))-E(0)\big)
&=-\frac{n}{2} \int_M(Q_{g(t)}-\overline{Q}_{g(t)})^2e^{nu(t)}dV_{g_\infty} \\
&\leq -c\int_M(Q_{g(t)}-\overline{Q}_{g(t)})^2e^{2nu(t)}dV_{g_\infty}\\
&=-c\big\|DE(u(t))\big\|_{L^2(M,g_\infty)}^2\\
&\leq -c\big|E(u(t))-E(0)\big|^{2-2\theta},
\end{split}
\end{equation}
where $c>0$ is a constant depending only on $n$ and $g_\infty$ (that we let change from line to line).
Let us first assume that the {\L}ojasiewicz-Simon inequality is satisfied with $\theta=\frac{1}{2}$, i.e.
we are in the integrable case. Then (\ref{1.7}) yields
$E(u(t))-E(0)\leq C e^{-2\delta t}$, for some $\delta>0$ depending only on $n$ and $g_\infty$,
and $C>0$ depending on $g(0)$ (chosen so that this actually holds for all $t\geq 0$).
On the other hand, if {\L}ojasiewicz-Simon inequality holds with $\theta\in (0,\frac{1}{2})$,
then the same argument shows that
$E(u(t))-E(0)\leq C(1+t)^{\frac{1}{2\theta-1}}$.

Exploiting the fact that the flow converges in $C^n$, we may use the
{\L}ojasiewicz-Simon inequality to compute
\begin{equation*}
\begin{split}
\frac{d}{dt}\big(E(u(t))-E(0)\big)^\theta
&=\theta\big(E(u(t))-E(0)\big)^{\theta-1}
\frac{d}{dt}\big(E(u(t))-E(0)\big)\\
&\leq -c\,\theta\big(E(u(t))-E(0)\big)^{\theta-1}\big\|DE(u(t))\big\|_{L^2(M,g_\infty)}^2\\
&\leq -c\,\theta\|DE_f(u(t))\big\|_{L^2(M,g_\infty)}\\
&\leq -c\,\theta\left\|\frac{\partial u(t)}{\partial t}\right\|_{L^2(M,g_\infty)}
\end{split}
\end{equation*}
where we have used  (\ref{1.8}) and (\ref{1.9}) in the last equality.
Thus, if $\theta=\frac{1}{2}$ (recall $\lim_{t\to\infty}u(t)=0$), then
\begin{equation*}
\begin{split}
\|u(t)\|_{L^2(M,g_\infty)} & \leq\int_t^\infty\left\|\frac{\partial u(s)}{\partial s}\right\|_{L^2(M,g_\infty)}ds \\
     & \leq -c\int_t^\infty\frac{d}{ds}\left[\big(E(u(s))-E(1)\big)^{\frac{1}{2}}\right]ds\\
     &=c\big(E(u(t))-E(1)\big)^{\frac{1}{2}}\leq Ce^{-\delta t}.
\end{split}
\end{equation*}
A similar computation if $\theta\in (0,\frac{1}{2})$
yields $\|u(t)\|_{L^2(M,g_\infty)}\leq C(1+t)^{-\frac{\theta}{1-2\theta}}$.

To obtain $C^n$ estimates, we may interpolate between $L^2(M,g)$ and $W^{k,2}(M,g)$
for $k$ large enough:
interpolation \cite[Theorem 6.4.5]{Bergh} and Sobolev embedding yields
some constant $\eta\in (0,1)$ so that
$$\|u(t)\|_{C^{n,\alpha}(M,g_\infty)}\leq \|u(t)\|_{L^2(M,g_\infty)}^\eta\|u(t)\|_{W^{k,2}(M,g_\infty)}^{1-\eta}.$$
Because $u(t)$ converges to $0$ in $C^{n,\alpha}$
(and thus in $C^\infty$ by parabolic Schauder estimates and bootstrapping),
the second term is uniformly bounded.
Thus, exponential (polynomial) decay of the $L^2$ norm gives
exponential (polynomial) decay of the $C^{n,\alpha}$ norm as well.
\end{proof}

\section{Slowly converging $Q$-curvature flow} \label{section4}

In this section, we show that, given a nonintegrable critical point $g_\infty$ satisfying a particular hypothesis, there exists a $Q$-curvature flow $g(t)$ such that $g(t)$ converges to $g_\infty$ exactly  at a polynomial rate.

This section is organized as follows: In section \ref{sec3.1}, we show that the $Q$-curvature flow can be represented by two different flows. To be more specific, we will project the flow equation to the kernel $\Lambda_0$ of $\mathcal{L}_\infty$ and its orthogonal complement $\Lambda_0^\perp$, respectively. In section \ref{sec3.2}, we solve the kernel-projected flow. In section \ref{sec3.3}, we solve the kernel-orthogonal projected flow. In section \ref{sec3.4}, we combine all the previous results to prove Theorem \ref{main2}.
\subsection{Projecting the $Q$-curvature flow with estimates}\label{sec3.1}

Here and in the sequel we will always use $f'(t)$ to denote the time derivative of a function $f(t)$. We will skip the proof of the following lemma, for its proof is the same as that of \cite[Lemma 15]{Chodosh}.

\begin{lem}\label{lem15}
  Assume that $g_\infty$ satisfies $AS_p$ as defined in Definition \ref{def10}, i.e. $F_p|_{\mathbb{S}^k}$ achieves a positive maximum for some point $\hat{v}$ in the unit sphere $\mathbb{S}^k\subset \Lambda_0$. Then, for any fixed $T\geq 0$, the function
  \begin{equation}\label{eq6}
    \varphi(t):=\varphi(t,T)=(T+t)^{-\frac{1}{p-2}}\left(\frac{n}{p(p-2)F_p(\hat{v})}\right)^\frac{1}{p-2}\hat{v}
  \end{equation}
  solves $n\varphi'+DF_p(\varphi)=0$.
\end{lem}

In the next result and subsequently in this section, we will always denote by $\|f(t)\|_{C^{k,\alpha}}$ the parabolic $C^{k,\alpha}$ norm on $(t,t+1)\times M$. More precisely, for $\alpha\in(0,1)$, we define the seminorm
\begin{equation*}
  |f(t)|_{C^{0,\alpha}}=\sup_{\substack{(s_i,x_i)\in(t,t+1)\times M \\ (s_1,x_1)\neq(s_2,x_2)}}\frac{|f(s_1,x_1)-f(s_2,x_2)|}{(d_{g_\infty}(x_1,x_2)^n+|t_1-t_2|)^\frac{\alpha}{n}}
\end{equation*}
and for $k\geq 0$ and $\alpha\in(0,1)$, we define the norm
\begin{equation}\label{eq8}
  \|f(t)\|_{C^{k,\alpha}}=\sum_{|\beta|+nj\leq k}\sup_{(t,t+1)\times M}|D^\beta_x D^j_t f|+\sum_{|\beta|+nj=k}|D^\beta_x D^j_t f|_{C^{0,\alpha}}
\end{equation}
where the norm and derivatives in the sum are taken with respect to $g_\infty$.
When we mean an alternative norm, we will always indicate the domain.

\begin{lem}\label{appdA}
For the functional $E$, there holds
\begin{equation}\label{eq27}
  \|D^3E(w)[u,v]\|_{C^{0,\alpha}}\leq C\|u\|_{C^{n,\alpha}}\|v\|_{C^{n,\alpha}}
\end{equation}
for some uniform constant $C>0$. Furthermore, for $w_1,w_2$ such that $\|w_i\|_{C^{n,\alpha}}<1$, we have
\begin{equation*}
\begin{split}
  \|D^3E(w_1)[v,v]-D^3E(w_2)[u,u]\|_{C^{0,\alpha}}\leq& C(\|w_1\|_{C^{n,\alpha}}+\|w_2\|_{C^{n,\alpha}})\\
&\qquad \times (\|u\|_{C^{n,\alpha}}+\|v\|_{C^{n,\alpha}})\|u-v\|_{C^{n,\alpha}}
\end{split}
\end{equation*}
for some uniform constant $C>0$.
\end{lem}
\begin{proof}
This follows from (\ref{A.1}).
\end{proof}

\begin{lem}\label{lemma16}
  There exists $T_0>0$, $\epsilon_0>0$ and $c>0$, all depending on $g_\infty$ and $\hat{v}$, such that the following holds: Fix $T>T_0$. Then, for $\varphi(t)$ as in Lemma \ref{lem15} and $w\in C^{n,\alpha}(M\times[0,\infty))$, and $u:=\varphi+w^\top+\Phi(\varphi+w^\top)+w^\perp$, the function
\begin{equation*}
  E_0^\top(w):=\mbox{\emph{proj}}_{\Lambda_0}\left[DE(u)e^{-nu}-DE(u)\right]
\end{equation*}
satisfies
\begin{equation*}
  \begin{split}
  &\left\|E_0^\top(w)\right\|_{C^{0,\alpha}}\leq c\left\{(T+t)^{-\frac{p-1}{p-2}}+\|w^\top\|^{p-1}_{C^{0,\alpha}}+\|w^\perp\|_{C^{n,\alpha}}\right\}\left\{(T+t)^{-\frac{1}{p-2}}+\|w\|_{C^{n,\alpha}}\right\},\\
&\left\|E^\top_0(w_1)-E_0^T(w_2)\right\|_{C^{0,\alpha}}\\
&\qquad\leq c\left\{(T+t)^{-\frac{p-1}{p-2}}+\|w_1^\top\|^{p-1}_{C^{0,\alpha}}+\|w_2^\top\|^{p-1}_{C^{0,\alpha}}+\|w_1^\perp\|_{C^{n,\alpha}}+\|w_2^\perp\|_{C^{n,\alpha}}\right\}\\
&\qquad \qquad \times \|w_1-w_2\|_{C^{n,\alpha}}\\
&+c\left\{(T+t)^{-\frac{1}{p-2}}+\|w_1\|_{C^{n,\alpha}}+\|w_2\|_{C^{n,\alpha}}\right\}\left(\|w_1^\top\|^{p-2}_{C^{0,\alpha}}+\|w_2^\top\|^{p-2}_{C^{0,\alpha}}\right)\\
&\qquad\qquad \times \|w_1^\top-w_2^\top\|_{C^{0,\alpha}}\\
&+c\left\{(T+t)^{-\frac{1}{p-2}}+\|w_1\|_{C^{n,\alpha}}+\|w_2\|_{C^{n,\alpha}}\right\}\|w_1^\perp-w_2^\perp\|_{C^{n,\alpha}}.
\end{split}
\end{equation*}
Identical estimates hold for $E_0^\perp(w):=\mbox{\emph{proj}}_{\Lambda_0^\perp}\left[DE(u)e^{-nu}-DE(u)\right]$. Here, we are using the parabolic H\"{o}lder norms on $(t,t+1)\times M$ as defined above; the bounds hold for each fixed $t\geq 0$, with the constants independent of $T$ and $t$.
\end{lem}
\begin{proof}
  First, we have
\begin{equation*}
  \begin{split}
  e^{-nu}&=1-n\int_0^1e^{-snu}u ds.
\end{split}
\end{equation*}
where $u=\varphi+w^\top+\Phi(\varphi+w^\top)+w^\perp$.
So, letting $E_0(w):=DE(u)e^{-nu}-DE(u)$, we have
\begin{equation}\label{eq9}
  \begin{split}
  &\|E_0(w)\|_{C^{0,\alpha}}\leq c\|DE(u)\|_{C^{0,\alpha}}\left((T+t)^{-\frac{1}{p-2}}+\|w^\top\|_{C^{0,\alpha}}+\|w^\perp\|_{C^{0,\alpha}}\right),
\end{split}
\end{equation}
where we used the fact that $\Phi(0)=0$ and $\Phi$ is an analytic map.
The rest of the proof is similar to that \cite[Lemma 16]{Chodosh} using  Taylor's theorem,  Proposition \ref{prop7}, and Lemma \ref{appdA}.
\end{proof}

The next result reduces the $Q$-curvature flow to two flows, one on $\Lambda_0$ and the other on $\Lambda_0^\perp$.

\begin{prop}\label{prop17}
  There exists $T_0>0$, $\epsilon_0>0$ and $c>0$, all depending on $g_\infty$ and $\hat{v}$, such that the following hods: Fix $T>T_0$. Then, for $\varphi(t)$ as in Lemma \ref{lem15} and $w\in C^{n,\alpha}(M\times [0,\infty))$, there are functions $E^\top(w)$ and $E^\perp(w)$ such that $u:=\varphi+w^\top+\Phi(\varphi+w^\top)+w^\perp$ is a solution to the $Q$-curvature flow if and only if
\begin{align}
  n(w^\top)'+D^2F_p(\varphi)w^\top=&E^\top(w),\label{eq11}\\
(w^\perp)'+\frac{1}{2}\mathcal{L}_\infty w^\perp=&E^\perp(w).\label{eq12}
\end{align}
Here, as long as $\|w\|_{C^{n,\alpha}}\leq \epsilon_0$, the error terms $E^\top$ and $E^\perp$ satisfy
\begin{equation*}
  \begin{split}
  \|E^\top(w)\|_{C^{0,\alpha}}\leq & c\left((T+t)^{-\frac{p-1}{p-2}}+\|w^\top\|^{p-1}_{C^{0,\alpha}}+\|w^\perp\|_{C^{n,\alpha}}\right)\left((T+t)^{-\frac{1}{p-2}}+\|w\|_{C^{n,\alpha}}\right)\\
&+c(T+t)^{-\frac{p}{p-2}}+c(T+t)^{-\frac{p-1}{p-2}}\|w^\top\|_{C^{0,\alpha}}+c(T+t)^{-\frac{p-3}{p-2}}\|w^\top\|^2_{C^{0,\alpha}}\\
&+c\|w^\top\|^{p-1}_{C^{0,\alpha}}+c\left((T+t)^{-\frac{1}{p-2}}+\|w\|_{C^{n,\alpha}}\right)\|w^\perp\|_{C^{n,\alpha}},
\end{split}
\end{equation*}
\begin{equation*}
  \begin{split}
 & \|E^\perp(w_1)-E^\perp(w_2)\|_{C^{0,\alpha}}\\
&\qquad \leq c\left((T+t)^{-\frac{p-1}{p-2}}+\|w_1^\top\|^{p-1}_{C^{0,\alpha}}+\|w_2^\top\|^{p-1}_{C^{0,\alpha}}+\|w_1^\perp\|_{C^{n,\alpha}}+\|w_2^\perp\|_{C^{n,\alpha}}\right)\\
&\qquad\qquad \times\|w_1-w_2\|_{C^{n,\alpha}}\\
&\qquad\quad +c\left((T+t)^{-\frac{1}{p-2}}+\|w_1\|_{C^{n,\alpha}}+\|w_2\|_{C^{n,\alpha}}\right)(\|w_1^\top\|^{p-2}_{C^{0,\alpha}}+\|w_2^\top\|^{p-2}_{C^{0,\alpha}})\\
&\qquad\qquad \times \|w_1^\top-w_2^\top\|_{C^{0,\alpha}}\\
&\qquad \quad +c\left((T+t)^{-\frac{1}{p-2}}+\|w_1\|_{C^{n,\alpha}}+\|w_2\|_{C^{n,\alpha}}\right)\|w_1^\perp-w_2^\perp\|_{C^{n,\alpha}}\\
&\qquad\quad +c\left((T+t)^{-\frac{p-3}{p-2}}(\|w_1^\top\|_{C^{0,\alpha}}+\|w_2^\top\|_{C^{0,\alpha}})+\|w_1^\top\|^{p-2}_{C^{0,\alpha}}+\|w_2^\top\|^{p-2}_{C^{0,\alpha}}\right)\\
&\qquad \qquad \times \|w_1^\top-w_2^\top\|_{C^{0,\alpha}}\\
&\qquad\quad +c(T+t)^{-\frac{p-1}{p-2}}\|w_1^\top-w_2^\top\|_{C^{0,\alpha}},
\end{split}
\end{equation*}
\begin{equation*}
  \begin{split}
&  \|E^\perp(w)\|_{C^{0,\alpha}}\\
&\qquad \quad \leq c\left((T+t)^{-\frac{p-1}{p-2}}+\|w^\top\|^{p-1}_{C^{0,\alpha}}+\|w^\perp\|_{C^{n,\alpha}}\right)\left((T+t)^{-\frac{1}{p-2}}+\|w\|^{C^{n,\alpha}}\right)\\
&\qquad\qquad +c\left((T+t)^{-\frac{1}{p-2}}+\|w\|_{C^{n,\alpha}}\right)\|w^\perp\|_{C^{n,\alpha}}\\
&\qquad\qquad +c\left((T+t)^{-\frac{1}{p-2}}+\|w\|_{C^{n,\alpha}}\right)\left((T+t)^{-\frac{p-1}{p-2}}+\|w'\|_{C^{0,\alpha}}\right)
\end{split}
\end{equation*}
\begin{equation*}
  \begin{split}
 & \|E^\perp(w_1)-E^\perp(w_2)\|_{C^{0,\alpha}}\\
&\qquad\quad \leq c\left((T+t)^{-\frac{p-1}{p-2}}+\|w_1^{\top}\|^{p-1}_{C^{0,\alpha}}+\|w_2^\top\|^{p-1}_{C^{0,\alpha}}+\|w_1^\perp\|_{C^{n,\alpha}}+\|w_2^\perp\|_{C^{n,\alpha}}\right)\\
&\qquad\qquad \times\|w_1-w_2\|_{C^{n,\alpha}}\\
&\qquad\quad +c\left((T+t)^{-\frac{1}{p-2}}+\|w_1\|_{C^{n,\alpha}}+\|w_2\|_{C^{n,\alpha}}\right)(\|w_1^\top\|^{p-2}_{C^{0,\alpha}}+\|w_2^\top\|^{p-2}_{C^{0,\alpha}})\\
&\qquad\qquad \times \|w_1^\top-w_2^\top\|_{C^{0,\alpha}}\\
&\qquad\quad +c\left((T+t)^{-\frac{1}{p-2}}+\|w_1\|_{C^{n,\alpha}}+\|w_2\|_{C^{n,\alpha}}\right)\|w_1^\perp-w_2^\perp\|_{C^{n,\alpha}}\\
&\qquad\quad +c\left((T+t)^{-\frac{1}{p-2}}+\|w_1\|_{C^{n,\alpha}}+\|w_2\|_{C^{n,\alpha}}\right)\|w_1'-w_2'\|_{C^{0,\alpha}}\\
&\qquad\quad +c\left((T+t)^{-\frac{p-1}{p-2}}+\|w_1'\|_{C^{0,\alpha}}+\|w_2'\|_{C^{0,\alpha}}\right)\|w_1-w_2\|_{C^{n,\alpha}}.
\end{split}
\end{equation*}
Here we are using the parabolic H\"{o}lder norms on $(t,t+1)\times M$ as defined above; the bounds hold for each fixed $t\geq 0$, with the constants independent of $T$ and $t$.
\end{prop}

\begin{proof}
We follow the proof of \cite[Proposition 17]{Chodosh}.
It follows from (\ref{1.8}) and (\ref{eq2}) that $u$ is a solution to the $Q$-curvature flow if and only if
\begin{equation}
\frac{\partial u}{\partial t}=-\frac{1}{n}DE(u)e^{-nu}.
\end{equation}
We now project the $Q$-curvature flow equation onto $\Lambda_0$ and $\Lambda_0^\perp$, so $u$ solves the $Q$-curvature flow if and only if the following two equations are satisfied:
\begin{equation}\label{decompose}
  \begin{split}
  &n(\varphi+w^\top)'=-\textrm{proj}_{\Lambda_0}\left[DE\left(\varphi+w^\top+\Phi(\varphi+w^\top)+w^\perp\right)\right]-E_0^\top(w),\\
  &n\left(\Phi(\varphi+w^\top)+w^\perp\right)'=-\textrm{proj}_{\Lambda_0^\perp}\left[DE\left(\varphi+w^\top+\Phi(\varphi+w^\top)+w^\perp\right)\right]-E_0^\perp(w),
\end{split}
\end{equation}
where $E_0(w)$ is defined as in Lemma \ref{lemma16}.
The rest of the proof is similar except for that the Taylor's theorem gives us
\begin{equation}\label{proj1}
  \begin{split}
  &\textrm{proj}_{\Lambda_0}DE\left(\varphi+w^\top+\Phi(\varphi+w^\top)+w^\perp\right)\\
&\qquad\qquad\qquad =\textrm{proj}_{\Lambda_0}DE\left(\varphi+w^\top+\Phi(\varphi+w^\top)\right)+E_1^\top(w),
\end{split}
\end{equation}
with slightly different bounds
\begin{equation}\label{bound1}
  \begin{split}
   \|E_1^\top(w)\|_{C^{0,\alpha}}\leq &c\left((T+t)^{-\frac{1}{p-2}}+\|w\|_{C^{n,\alpha}}\right)\|w^\perp\|_{C^{n,\alpha}},\\
\|E_1^\top(w_1)-E_1^\top(w_2)\|_{C^{0,\alpha}}\leq &c\left(\|w_1\|_{C^{n,\alpha}}+\|w_2\|_{C^{n,\alpha}}\right)\|w_1^\perp-w_2^\perp\|_{C^{n,\alpha}}.
\end{split}
\end{equation}

\end{proof}

\bigskip
\subsection{Solving the kernel-projected flow with polynomial decay estimates}\label{sec3.2}

In this subsection we solve the kernel-projected flow (\ref{eq11}). First, from the definition of $\varphi$ in (\ref{eq6}) and the fact that $D^2F_p$ is homogeneous of degree $p-2$,
\begin{equation*}
  D^2F_p(\varphi)=(T+t)^{-1}\underbrace{\left(\frac{n}{p(p-2)F_p(\hat{v})}\right)D^2F_p(\hat{v})}_{:=\mathcal{D}}.
\end{equation*}
Let $\mu_1,\cdots,\mu_k$ be the eigenvalues of $\mathcal{D}$ and $e_i$ the corresponding orthonormal basis in which $\mathcal{D}$ is diagonalized. Then the kernel-projected flow is equivalent to the following system of ODEs for $v_i:=w^\top\cdot e_i$,
\begin{equation}\label{eq14}
  n(w^\top)'+\frac{\mu_i}{T+t}v_i=E_i^\top:=E^\top\cdot e_i,\quad i=1,\cdots,k.
\end{equation}
Fix for the rest of this subsection a number $\gamma$ with $\gamma\notin\left\{\frac{1}{n}\mu_1,\cdots,\frac{1}{n}\mu_k\right\}$.
Define the following weighted norms:
\begin{equation*}
  \|u\|_{C^{0,\alpha}_\gamma}:=\sum_{t>0}\left[(T+t)^\gamma\|u(t)\|_{C^{0,\alpha}}\right] \ \textrm{ and } \ \|u\|_{C^{0,\alpha}_{1,\gamma}}:=\|u\|_{C^{0,\alpha}_\gamma}+\|u'\|_{C^{0,\alpha}_{1+\gamma}}.
\end{equation*}
We recall that these H\"{o}lder norms are space-time norms on the interval $(t,t+1)\times  M$, as defined in (\ref{eq8}).

Given $\gamma$ as above, we define $\Pi_0=\Pi_0(\gamma)$ by
\begin{equation}
  \Pi_0:=\textrm{span}\left\{v\in\Lambda_0:\mathcal{D}v=\mu v,\textrm{ and }\mu>n\right\}.
\end{equation}
Moreover, let $\textrm{proj}_{\Pi_0}:\Lambda\rightarrow \Pi_0$ be the corresponding linear projector.

\begin{lem}\label{lemma18}
  For any $T>0$ such that $\|E^\top\|_{C^{0,\alpha}_{1+\gamma}}<\infty$, there is a unique $u$ with $u(t)\in\Lambda_0$, $t\in[0,\infty)$, satisfying $\|u\|_{C^0_\gamma}<\infty$, $\textrm{\emph{proj}}_{\Pi_0}(u(0))=0$, and such that $v_i:=u\cdot e_i$ solves the system \eqref{eq14}. Furthermore, we have the bound
\begin{equation*}
  \|u\|_{C^{0,\alpha}_{1,\gamma}}\leq C\|E^\top\|_{C^{0,\alpha}_{1+\gamma}}.
\end{equation*}
\end{lem}
\begin{proof}
See \cite[Lemma 18]{Chodosh}.
\end{proof}
\subsection{Solving the kernel-orthogonal projected flow}\label{sec3.3}
Define the weighted norms
\begin{equation*}
  \|u\|_{L^2_q}=\sup_{t\in[0,\infty)}\left[(T+t)^q\|u(t)\|_{L^2( M)}\right],
\end{equation*}
where the $L^2$ norm is the spatial norm of $u(t)$ on $ M$, taken with respect to $g_\infty$, and
\begin{equation*}
  \|u\|_{C^{n,\alpha}_q}=\sup_{t\geq 0} \left[(T+t)^q\|u(t)\|_{C^{n,\alpha}}\right],
\end{equation*}
where, as usual, the H\"{o}lder norms are the space-time norms defined in (\ref{eq8}). Also, let
\begin{equation*}
  \begin{split}
  &\Lambda_\downarrow:=\overline{\textrm{span}\{\varphi\in C^\infty(M):\mathcal{L}_\infty \varphi+\delta \varphi=0,\delta>0\}}^{L^2},\\
&\Lambda_\uparrow :=\textrm{span}\{\varphi\in C^\infty(M):\mathcal{L}_\infty \varphi+\delta \varphi=0,\delta<0\}.
\end{split}
\end{equation*}

From the spectral theory, $L^2( M,g_\infty)=\Lambda_\uparrow\oplus\Lambda_0\oplus \Lambda_\downarrow$ and $\Lambda_\uparrow$ and $\Lambda_0$ are finite dimensional. Write the nonnegative integers as an ordered union $\mathbb{N}=K_\uparrow\cup K_0\cup K_\downarrow$, where the ordering of the indices comes from an ordering of the eigenfunctions of the $\mathcal{L}_\infty$ and the partitioning of $\mathbb{N}$ corresponds to which of $\Lambda_\downarrow$, $\Lambda_0$, or $\Lambda_\uparrow$ the $k$-th eigenfunction of $\mathcal{L}_\infty$ lies in.

\begin{lem}\label{lemma19}
  For any $T>0$ and $q<\infty$ such that $\|E^\perp\|_{L^2_q}<\infty$, there is a unique $u(t)$ with $u(t)\in\Lambda_0^\perp$, $t\in[0,\infty)$, satisfying $\|u\|_{L^2_q}<\infty$, $\textrm{\emph{proj}}_{\Lambda_\downarrow}(u(0))=0$ and
\begin{equation}\label{eq15}
  u'=\mathcal{L}_\infty u+E^\perp.
\end{equation}
Furthermore, $\|u\|_{L^2_q}\leq C\|E^\perp\|_{L^2_q}$ and $\|u\|_{C^{n,\alpha}_q}\leq C\|E^\perp\|_{C^{0,\alpha}_q}$.
\end{lem}
\begin{proof}

We now consider the $C^{n,\alpha}_q$ bounds for $u$. By interior parabolic Schauder estimates (see \cite{Dong&Zhang} for example) and Arzel\`{a}-Ascoli theorem, we have that for $t\geq 1$,
\begin{equation*}
  \begin{split}
  &\|u(t)\|_{C^{n,\alpha}}\leq C\left(\sup_{s\in(t-1,t+1)}\|u(s,x)\|_{L^2(M)}+\|E^\perp\|_{C^{0,\alpha}((t-1,t+1)\times  M)}\right)\\
&\hspace{6cm}  +C\epsilon \|u(t)\|_{C^{0,\alpha}((t-1,t+1)\times  M)}.
\end{split}
\end{equation*}
Multiplying it by $(T+t)^q$ and taking the supremum over $t\geq1$ yields
\begin{equation}\label{4.3.1}
\begin{split}
 & \sup_{t\geq 1}\left[(T+t)^q\|u(t)\|_{C^{n,\alpha}}\right]\leq C\|E^\perp\|_{C^{0,\alpha}_q}+C\epsilon\|u\|_{C^{0,\alpha}_q}
\end{split}
\end{equation}
using  $\|u\|_{L^2_q}\leq C\|E^\perp\|_{L^2_q}$.
To finish the proof, it remains to extend the supremum up to $t=0$. The global Schauder estimates (c.f. \cite{Dong&Zhang}) shows that
\begin{equation}\label{4.3.2}
  \begin{split}
  \|u(t)\|_{C^{n,\alpha}((0,1)\times M)}&\leq C\Big(\sup_{s\in(0,1)}\|u(s,x)\|_{L^2( M)}+\epsilon\|u\|_{C^{0,\alpha}((0,1)\times M)}\\
&\qquad\qquad +\|E^\perp\|_{C^{0,\alpha}((0,1)\times  M)}+\|u(0)\|_{C^{n,\alpha}(M)}\Big).
\end{split}
\end{equation}
The remaining parts of proof is analogous to that of \cite[Lemma 19]{Chodosh}.
\end{proof}
\subsection{Construction of a slowly converging flow}\label{sec3.4}

To proceed further, we define the norm
\begin{equation*}
  \|f\|_{\gamma}^*:=\|\textrm{proj}_{\Lambda_0}f\|_{C^{0,\alpha}_{1,\gamma}}+\|\textrm{proj}_{\Lambda_0^\perp}f\|_{C^{n,\alpha}_{1+\gamma}}.
\end{equation*}
For $\gamma$ to be specified below,  the Banach space $X$ is defined as
\begin{equation}\label{defX}
  X:=\{w:\|w\|_{\gamma}^*<\infty\}.
\end{equation}

\begin{prop}\label{prop4.7}
  Assume that $g_\infty$ satisfies $AS_p$. We may thus fix a point where $F_p|_{\mathbb{S}^{k-1}}$ achieves a positive maximum and denote it by $\hat{v}$. Define
\begin{equation*}
  \varphi(t)=(T+t)^{-\frac{1}{p-2}}\left(\frac{n}{p(p-2)F_p(\hat{v})}\right)^{\frac{1}{p-2}}\hat{v}
\end{equation*}
as in Lemma \ref{lem15}. Then, there exists $C>0$, $T>0$, $\frac{1}{p-2}<\gamma<\frac{2}{p-2}$ and $u(t)\in C^\infty(M\times (0,\infty))$ such that
  $u(t)>0$ for all $t>0$, $g(t):=e^{2u(t)}g_\infty$ is a solution to the $Q$-curvature flow and
\begin{equation*}
  \|w^\top(t)+\Phi\left(\varphi(t)+w^\top(t)\right)+w^\perp(t)\|_\gamma^*=\|u(t)-\varphi(t)\|_\gamma^*\leq C.
\end{equation*}
\end{prop}
\begin{proof}
The proposition follows from the contraction mapping principle, as applied \cite[Proposition 20]{Chodosh}.
\end{proof}

Now we are ready to prove Theorem \ref{main2}.
\begin{proof}[Proof of Theorem \ref{main2}]
  From Proposition \ref{prop4.7}, we have constructed $\varphi(t)$ and $u(t)$ so that
\begin{equation*}
  \varphi(t)=(T+t)^{-\frac{1}{p-2}}\left(\frac{n}{p(p-2)F_p(\hat{v})}\right)^{\frac{1}{p-2}}\hat{v},
\end{equation*}
 $e^{2u(t)}g_\infty$ is a solution to the $Q$-curvature flow, and
\begin{equation*}
  u(t)=\varphi(t)+\tilde{w}(t):=\varphi(t)+w^\top(t)+\Phi\left(\varphi(t)+w^\top(t)\right)+w^\perp(t),
\end{equation*}
where $\tilde{w}(t)$ satisfies $\|\tilde{w}\|_{C^0}\leq C(1+t)^{-\gamma}$ for some $C>0$ and all $t\geq 0$. We have arranged that $\gamma>1/(p-2)$, which implies that $\varphi(t)$ is decaying slower than $\tilde{w}(t)$. Thus
\begin{equation*}
  \|u(t)\|_{C^0}\geq C(1+t)^{-\frac{1}{p-2}}
\end{equation*}
as $t\rightarrow \infty$. From this, the assertion follows
\end{proof}

\section{An example satisfying $AS_3$}\label{section5}

In this section, we provide an example of Riemannian metric satisfying $AS_3$ for the 6th order Paneitz operator $P_6$ and $Q_6$-curvature in dimension $n = 6$.
Moreover, we will see that this Riemannian metric satisfies the assumptions in Theorem \ref{Brendle_thm}.
 This allows us  to conclude the existence of slowly converging $Q$-curvature flows.

We will omit the subscript $g$ in curvature tensors and differential operators in this section, as it will not cause any ambiguity. We introduce some notations: $J  = \frac{R}{10}$, $\delta$ denotes the adjoint operator of exterior derivative $d$, and $A$ denotes the Schouten tensor $A = \frac{1}{4}(Ric - Jg)$. The Bach tensor $\mathcal{B}$ is defined as
\begin{equation}\label{bach}
\mathcal{B}_{ij} = \Delta A_{ij} - \nabla^k \nabla_j A_{ik} + A^{kl} W_{ikjl}
\end{equation}

We recall following formulas for the $P_6$ and $Q_6$ in dimension 6 from \cite[Theorem 6.10.3]{Juhl1} and \cite[Theorem 10.1]{Juhl2}.  Note that the sign of $P_6$ differs from \cite{Juhl2} because we want $P_6$ to have $(-\Delta)^3$ as the leading order term. Additionally, the sign of the Riemann curvature tensor is different from \cite{Juhl1}.
\begin{align*}
P_6 =&  -\Delta^3 - 4\Delta \delta (J-2A) d - 4 \delta (J-2A) d \Delta - 2 \Delta (J\Delta)  \\
&- 8\delta (-J^2 + |A|^2 +4JA - 6A^2 - \mathcal{B} ) d
\end{align*}
\begin{align*}
Q_6 = &  \big\{ 16tr(A^3) - 24J|A|^2 + 8J^3 + 8\langle \mathcal{B}, A \rangle \big\} \\
&  + (\Delta^2 J - 8\delta (AdJ) ) + 4 \Delta(|A|^2  -J^2 )
\end{align*}

Our goal is to find a $\lambda$ such that $(S^2(\lambda) \times \mathbb{C}P^2, g_{prod})$ with the product metric $g_{prod}$ satisfies the $AS_3$ condition, where $S^2(\lambda)$ is the standard sphere with metric $g_\lambda$ having the sectional curvature $\lambda$ and $\mathbb{C}P^2$ is equipped with the Fubini-Study metric $g_{FS}$.

There is a canonical projection from $S^2(\lambda) \times \mathbb{C}P^2$ to each of $S^2(\lambda)$ and $\mathbb{C}P^2$ denoted by $\pi_1$ and $\pi_2$, respectively.  We denote the pullback of the metric and the curvature quantities by using the superscript (1) and (2). For example, $\gone = \pi_1^* g_\lambda$, $\gtwo = \pi_2^* g_{FS}$, $Rm^{(1)} = \pi_1^* Rm_{g_\lambda}$ and $Rm^{(2)} = \pi_2^* Rm_{g_{FS}}$. We also denote $\Delta^{(1)}, \Delta^{(2)}$ Laplacians acting on the first component and the second component of the product structure, respectively.

It is clear that
$$g_{prod} = \pi_1^* g_\lambda \oplus \pi_2^* g_{FS} = \gone \oplus \gtwo$$
$$Rm = \pi_1^* Rm_{g_\lambda} \oplus \pi_2^* Rm_{g_{FS}} = Rm^{(1)} \oplus Rm^{(2)}$$
$$\Delta = \Delta^{(1)} + \Delta^{(2)}.$$
Since $(\mathbb{C}P^2, g_{FS})$ is an Einstein manifold with Einstein constant 6, we have
$$Ric = \lambda \gone \oplus 6 \gtwo, R = 24 +2\lambda$$
$$J =\frac{\lambda + 12}{5}, A = \big(\frac{\lambda-3}{5} \big)\gone \oplus \big(\frac{18-\lambda}{20}\big) \gtwo.$$

We recall the following decomposition of the Riemann curvature tensor:
$$Rm = W + A\owedge g$$
where $\owedge$ is the Kulkarni-Nomizu product defined as $(S\owedge T)_{ijkl} = S_{ik} T_{jl} + S_{jl} T_{ik} - S_{il} T_{jk} - S_{jk} T_{il}$. We observe the following basic formulas for the product metric.

\begin{lem}\label{owedge} The following identities hold for components $g^{(1)}$ and $g^{(2)}$ of $g_{prod}$.

(1)  $(g^{(1)})^{kl} (g^{(1)} \owedge g^{(1)})_{ikjl} = 2 g^{(1)}_{ij}$.

(2) $(g^{(2)})^{kl} (g^{(2)} \owedge g^{(2)})_{ikjl} = 6 g^{(2)}_{ij}$.

(3) $(g^{(1)})^{kl} (g^{(1)} \owedge g^{(2)})_{ikjl} = 2 g^{(2)}_{ij}$.

(4) $(g^{(2)})^{kl} (g^{(1)} \owedge g^{(2)})_{ikjl} = 4 g^{(1)}_{ij}$.
\end{lem}
\begin{proof}
We prove the first and third identity by observing $(\gone)^{kl} \gtwo_{ik} = 0$. For the first identity:
$$(g^{(1)})^{kl} (g^{(1)} \owedge g^{(1)})_{ikjl} =2 (g^{(1)})^{kl}\big( \gone_{ij} \gone_{kl} - \gone_{il}\gone_{jk} \big)  = (4-2) g^{(1)}_{ij} = 2g^{(1)}_{ij}. $$

For the third identity:
$$(g^{(1)})^{kl} (g^{(1)} \owedge g^{(2)})_{ikjl} = (g^{(1)})^{kl} \big (\gone_{ij} \gtwo_{kl} + \gone_{kl} \gtwo_{ij} - \gone_{il}\gtwo_{jk}  - \gone_{jk}\gtwo_{il}  \big) = 2 g^{(2)}_{ij}.  $$

\end{proof}

Now we are ready to compute $Q_6$-curvature and $P_6$ operator for the product metric $g$	.
\begin{prop}
For $(S^2(\lambda) \times \mathbb{C}P^2, g_{prod})$, we have
\begin{align*}
Q_6 =&  0.48 \lambda^3 - 6.72 \lambda^2 + 63.36 \lambda - 34.56\\
P_6 = & -\Delta^3 -2(\lambda -12) \Delta^{(1)} \Delta +2\lambda \Delta^{(2)} \Delta \\
& + 8\big(- \frac{9}{50} \lambda^2 + \frac{42}{25} \lambda - \frac{198}{25} \big) \Delta^{(1)} + 8\big(\frac{7}{100}\lambda^2 - \frac{33}{25} \lambda +\frac{27}{25} \big) \Delta^{(2)}.
\end{align*}
\end{prop}

\begin{proof}

Due to the product structure of the metric $g_{prod}$, it is easy to see that the formula for $Q_6$-curvature is simplified as
\begin{equation}\label{Q6}
Q_6 =  6tr(A^3) - 24J|A|^2 + 8J^3 + 8\langle \mathcal{B}, A \rangle.
\end{equation}
It remains for us to compute the Bach tensor $\mathcal{B}$. Using the formula (\ref{bach}) and noting that both $(S^2(\lambda), g_\lambda)$ and $(\mathbb{C}P^2, g_{FS})$ are Einstein, it is easy to see that $\mathcal{B}_{ij} = A^{kl}W_{ikjl}.$

To compute $A^{kl}W_{kijl}$, we use the identity
\begin{align*}
W + A\owedge g & = Rm = Rm^{(1)} \oplus Rm^{(2)} \\
& = \big(\frac{\lambda}{2} \gone \owedge \gone \big) \oplus \big(W^{(2)} + A^{(2)} \owedge \gtwo \big)
\end{align*}
to see that
\begin{align*}
\mathcal{B}_{ij} &= A^{kl}W_{ikjl} \\
&= A^{kl} \big[-\big(A\owedge g \big) \oplus \big(\frac{\lambda}{2} \gone \owedge \gone \big) \oplus \big(W^{(2)}  + A^{(2)} \owedge \gtwo \big) \big]_{ikjl} \\
& =  A^{kl} \big[-\big(A\owedge g \big) \oplus \big( \frac{\lambda}{2} \gone \owedge \gone \big) \oplus \big(   \gtwo \owedge \gtwo\big) \big]_{ikjl} \\
& =  \big(\frac{\lambda+2}{20}\big) \big[ \big(\frac{\lambda-3}{5} \big)\gone \oplus \big(\frac{18-\lambda}{20}\big) g^2 \big]^{kl} \big( 6 \gone \owedge \gone \oplus \gtwo \owedge \gtwo - 3 \gone \owedge \gtwo \big)_{ikjl} \\
& =\frac{3(\lambda+2)(\lambda-6)}{40}(2\gone - \gtwo)
\end{align*}
where we used identities from Lemma \ref{owedge}.

Now it is straightforward to see that:
\begin{align*}
Q_6 = &  \big\{ 16tr(A^3) - 24J|A|^2 + 8J^3 + 8\langle \mathcal{B}, A \rangle \big\} \\
 = & 16\big[ 2 \big( \frac{\lambda-3}{5}\big)^3 + 4 \big(\frac{18-\lambda}{20} \big)^3  \big] - \frac{24(\lambda+12)}{5} \big[ 2 \big( \frac{\lambda-3}{5}\big)^2 + 4 \big(\frac{18-\lambda}{20} \big)^2 \big] \\
& + 8\frac{(\lambda+12)^3}{5^3} + \frac{3}{5} (\lambda+2) (\lambda-6)^2 \\
= & 0.48 \lambda^3 - 6.72 \lambda^2 + 63.36 \lambda - 34.56
\end{align*}

Next, we compute the $P_6$ operator. It is obvious that $\Delta^{(1)}$ and $\Delta^{(2)}$ commute.
\begin{align*}
P_6 =&  -\Delta^3 - 4\Delta \delta (J-2A) d - 4 \delta (J-2A) d \Delta - 2 \Delta (J\Delta)  \\
&- 8\delta (-J^2 + |A|^2 +4JA - 6A^2 - \mathcal{B} ) d \\
= & -\Delta^3 -2(\lambda -12) \Delta^{(1)} \Delta +2\lambda \Delta^{(2)} \Delta \\
& + 8\big(- \frac{9}{50} \lambda^2 + \frac{42}{25} \lambda - \frac{198}{25} \big) \Delta^{(1)} + 8\big(\frac{7}{100}\lambda^2 - \frac{33}{25} \lambda +\frac{27}{25} \big) \Delta^{(2)}.
\end{align*}
\end{proof}

\begin{exmp}\label{example}
We show that there exists $\lambda>0$ such that $P_6 \phi_1 = 6 Q_6 \phi_1$ where $\phi_1 = z_1 \bar{z}_2 + z_2 \bar{z}_1 + z_2 \bar{z}_3 + z_3 \bar{z}_2 + z_3 \bar{z}_1 + z_1 \bar{z}_3 $ is the first eigenfunction of the Laplacian of $(\mathbb{C}P^2, g_{FS})$ with eigenvalue is 12. By using expressions computed above,  the equation we need to solve is:
$$12^3 +2\lambda \cdot 144 -  8\big(\frac{7}{100}\lambda^2 - \frac{33}{25} \lambda +\frac{27}{25} \big) \cdot 12 = 6  0.48 \lambda^3 - 6.72 \lambda^2 + 63.36 \lambda - 34.56.$$
Solving this equation numerically, we find a positive real root $\lambda_0 \approx 15.206$.

It is known that $\int_{\mathbb{C}P^2} \phi_1^3 \ne 0$ by \cite{Kroncke}. Using Fubini's theorem,  we get
$$\int_{S^2(\lambda_0) \times \mathbb{C}P^2} (1\otimes \phi_1)^3  dV_g = Vol(S^2(\lambda_0)) \int_{\mathbb{C}P^2} \phi_1^3 \ne 0.$$

 Thus, $(S^2(\lambda_0) \times \mathbb{C}P^2, g_{prod})$ is degenerate and the function $F_3$ is not everywhere zero on $\Lambda_0$. This demonstrates that $(S^2(\lambda_0) \times \mathbb{C}P^2, g_{prod})$ satisfies $AS_3$ condition.

Next, we prove that for such $\lambda_0$, $P_6 \ge 0$ with the kernel consisting of constant functions. We compute:
\begin{align*}
P_6 =&  -\Delta^3 -2(\lambda -12) \Delta^{(1)} \Delta +2\lambda \Delta^{(2)} \Delta \\
& + 8\big(- \frac{9}{50} \lambda^2 + \frac{42}{25} \lambda - \frac{198}{25} \big) \Delta^{(1)} + 8\big(\frac{7}{100}\lambda^2 - \frac{33}{25} \lambda +\frac{27}{25} \big) \Delta^{(2)} \\
=&-\Delta^3 -2(\lambda -12) (\Delta^{(1)})^2  +24  \Delta^{(1)} \Delta^{(2)} + 2\lambda (\Delta^{(2)})^2  \\
& + 8\big(- \frac{9}{50} \lambda^2 + \frac{42}{25} \lambda - \frac{198}{25} \big) \Delta^{(1)} + 8\big(\frac{7}{100}\lambda^2 - \frac{33}{25} \lambda +\frac{27}{25} \big) \Delta^{(2)}\\
\ge & -(\Delta^{(1)})^3 - (\Delta^{(2)})^3 -2(\lambda -12) (\Delta^{(1)})^2  + 2\lambda (\Delta^{(2)})^2  \\
& + 8\big(- \frac{9}{50} \lambda^2 + \frac{42}{25} \lambda - \frac{198}{25} \big) \Delta^{(1)} + 8\big(\frac{7}{100}\lambda^2 - \frac{33}{25} \lambda +\frac{27}{25} \big) \Delta^{(2)} \\
= & -\Delta^{(1)}\big[  (\Delta^{(1)})^2 + 2(\lambda -12) \Delta^{(1)} -8\big(- \frac{9}{50} \lambda^2 + \frac{42}{25} \lambda - \frac{198}{25} \big) \big] \\
& -\Delta^{(2)} \big[  (\Delta^{(2)})^2 - 2\lambda \Delta^{(2)} - 8\big(\frac{7}{100}\lambda^2 - \frac{33}{25} \lambda +\frac{27}{25} \big) \big].
\end{align*}
\end{exmp}

We  observe that  $(2\lambda_0)^2 - 2(\lambda_0 -12) 2\lambda_0  -8\big(- \frac{9}{50} \lambda_0^2 + \frac{42}{25} \lambda_0 - \frac{198}{25} \big) > 0$  and $12^2 -2\lambda_0 * 12 - 8\big(\frac{7}{100}\lambda_0^2 - \frac{33}{25} \lambda_0 +\frac{27}{25} \big) >0$. Hence $P_6 \ge 0$ and it is easy to see that the kernel consists of constant functions.

The total $Q_6$-curvature is
\begin{align*}
\int_{S^2(\lambda_0) \times \mathbb{C}P^2} Q_6 dV_g & = \big( 0.48 \lambda_0^3 - 6.72 \lambda_0^2 + 63.36 \lambda_0 - 34.56\big)\cdot Vol(S^2(\lambda_0))\cdot Vol(\mathbb{C}P^2 )\\
& = \big( 0.48 \lambda_0^3 - 6.72 \lambda_0^2 + 63.36 \lambda_0 - 34.56\big) \cdot \frac{4\pi}{\lambda_0} \cdot \frac{\pi^2}{2} \\
& < 144\pi^3
\end{align*}
where $144\pi^3$ is the total $Q_6$-curvature of the standard sphere $S^{6}$. This confirms that the example satisfies the conditions of Theorem \ref{Brendle_thm}.

\section{Comparison of Cases for $n=2, 4$ and $6$.}\label{section6}

In this section, we briefly discuss some differences in low-dimensional cases in terms of whether the operator $\mathcal{L}_\infty$ has a non-trivial kernel.

When $n=2$, i.e. $(M,g)$ is a closed Riemannian manifold of dimension $2$,
the $Q$-curvature $Q_g$ is simply the Gauss curvature $K_g$.
Moreover, the Paneitz operator $P_g$ is just but the
negative Laplacian, ie. $P_g=-\Delta_g$.
Therefore, the $Q$-curvature flow reduces to the two dimensional normalized Ricci flow:
\begin{equation}\label{Gauss_flow}
\frac{\partial}{\partial t}g(t)=-(K_{g(t)}-\overline{K}_{g(t)})g(t).
\end{equation}

The operator $\mathcal{L}_\infty$ is given by
$$\mathcal{L}_\infty=-\Delta_{g_\infty}-2K_{g_\infty}.$$

If $K_{g_\infty}\le 0$, then $\ker\mathcal{L}_{g_\infty}$ is trivial. 
If $K_{g_\infty} \equiv \text{constant}> 0$, then $(M, g)$ is either the standard sphere or the real projective plane. The standard sphere is integrable by Obata's theorem and for the real projective plane, $\mathcal{L}_\infty$  is positive definite as the first eigenfunctions of Laplacian on the sphere are odd functions.  Hence, we can easily observe that the 2-dimensional $Q$-curvature flow  always converges exponentially.

For $n=4$, we have the following explicit formulas for the fourth-order Paneitz operator and the $Q_4$-curvature for $(M, g_{\infty})$:
$$(Q_4)_{g_{\infty}}  = \frac{1}{6}(-\Delta_{g_{\infty}} R_{g_{\infty}} + R_{g_{\infty}}^2 - 3|Ric_{g_{\infty}}|^2) $$
$$ (P_4)_{g_{\infty}} = \Delta_{g_{\infty}}^2  + \delta(\frac{2}{3}R_{g_{\infty}} g_\infty - 2Ric_{g_{\infty}} )d. $$
 Whether the operator $\mathcal{L}_\infty= (P_4)_{g_\infty} - 4(Q_4)_\infty$ is integrable or non-negative is unknown. This is related to the uniqueness of solutions to the constant $Q$-curvature equation. If one proves that, for a Riemannian manifold   $(M, g)$ with  $(Q_4)_g \equiv\text{constant} >0$ and $R_g>0$, the inequality
\begin{equation}\label{P4inequality}
\int_M (P_4)_g \phi \cdot \phi dVol_g \ge 4(Q_4)_g \int_M \phi^2 dVol_g  \text{  for} \int_M \phi = 0
\end{equation}
is true with equality holding if and only if $\phi \equiv 0$ or $(M, g)$ is the standard round sphere,  then the solution for constant $Q_4$-curvature equation with the constraint that the scalar curvature is positive is unique up to constants by the degree theory, as observed in \cite[Section 5]{Lee}. The uniqueness problem of the constant $Q_4$ -curvature equation is only solved when $g_\infty$ is Einstein with a positive Einstein constant by \cite{Vetois}, and the inequality (\ref{P4inequality}) is true by the eigenvalue comparison theorem. Hence the $Q_4$-curvature flow is converging exponentially to $g_\infty$ when $g_\infty$ is Einstein with a positive Einstein constant.

Example \ref{example} indicates that we cannot expect such phenomena in dimension $n=6$. Specifically, for $(S^2(\lambda_0) \times \mathbb{C}P^2, g_{prod})$, it is easy to see that 
$$R_{g_{prod}} > 0$$
$$(Q_4)_{g_{prod}} =-\Delta_{g_{prod}} J_{g_{prod}} - 2|A_{g_{prod}}|^2 +3 J_{g_{prod}}^2 = -0.06 \lambda_0^2 + 4.56 \lambda_0 + 936 > 0$$
$$(Q_6)_{g_{prod}} > 0$$
while the operator $P_6 - 6Q_6$ has a non-constant function in the kernel.

\section{Acknowledgement}

The first author was supported  by the National Science and Technology Council (NSTC),
Taiwan, with grant number: 112-2115-M-032-006-MY2. The second author was supported by KIAS Individual Grant, with grant number: MG096101.

\bibliographystyle{amsplain}

\end{document}